\documentclass[a4paper,11pt]{article}

\usepackage[top=2.5cm,bottom=2.5cm]{geometry}
\usepackage{amsmath}
\usepackage{amsfonts}
\usepackage{amssymb}
\usepackage{amsthm}

\newtheorem{theorem}{Theorem}[section]
\newtheorem{lemma}[theorem]{Lemma}
\newtheorem{corollary}[theorem]{Corollary}
\newtheorem{proposition}[theorem]{Proposition}
\newtheorem{conjecture}[theorem]{Conjecture}

\theoremstyle{definition}
\newtheorem{definition}[theorem]{Definition}
\newtheorem{examples}[theorem]{Examples}
\newtheorem{remark}[theorem]{Remark}
\newtheorem{remarks}[theorem]{Remarks}

\newcommand{\N}{\mathbb{N}}

\newcommand{\R}{\mathbb{R}}

\newcommand{\p}{\mathfrak{p}}
\newcommand{\enum}[1]{#1_{1}, \ldots, #1_{n}}
\newcommand{\enumi}[2]{#1_{1}, \ldots, #1_{#2}}
\newcommand{\set}[2]{\{ #1 \mid #2 \}}
\newcommand{\bqf}[2]{\langle #1, #2 \rangle}
\newcommand{\OR}{\mathcal{O}}
\newcommand{\m}{\mathfrak{m}}

\newcommand{\Quot}[1]{\textrm{Quot}(#1)}
\newcommand{\Sper}[1]{\textrm{Sper} \, #1}
\newcommand{\supp}[1]{\textrm{supp}(#1)}
\newcommand{\cent}[1]{\textrm{cent}(#1)}
\newcommand{\SA}[1]{\textrm{SA}(#1)}
\newcommand{\PW}[1]{\textrm{PW}(#1)}
\newcommand{\ord}[1]{\textrm{ord}_{#1}}

\title{On the Pierce-Birkhoff Conjecture for Smooth Affine Surfaces over Real Closed Fields
\thanks{submitted for publication on March 31, 2008}}

\author{Sven Wagner}

\date{}

\begin{document}

\maketitle

\begin{abstract}

	\noindent We will prove that the Pierce-Birkhoff Conjecture holds for non-singular two-dimen\-sional affine real algebraic varieties over real closed fields, i.e., if $W$ is such a variety, then every piecewise polynomial function on $W$ can be written as suprema of infima of polynomial functions on $W$. More precisely, we will give a proof of the so-called Connectedness Conjecture for the coordinate rings of such varieties, which implies the Pierce-Birkhoff Conjecture.

\end{abstract}

\section{Introduction}

In 1956, G.~Birkhoff and R.~S.~Pierce raised the following question, which is well-known today as the \textbf{Pierce-Birkhoff Conjecture}:

\begin{conjecture}{(Pierce, Birkhoff)}

	Let $t\colon \R^{n} \rightarrow \R$ be a piecewise polynomial function. Then there is a finite family of polynomials $\{ h_{ij} \}_{i \in I,j \in J} \subset \R[\enum{X}]$ such that
	\[ t = \sup_{i \in I} \big{(} \inf_{j \in J} h_{ij} \big{)}. \]

\end{conjecture}

The statement of this conjecture depends on the following

\begin{definition}

	A function $t\colon \R^{n} \rightarrow \R$ is said to be piecewise polynomial if there are closed semialgebraic subsets $\enumi{U}{m}$ of $\R^{n}$ and polynomials $\enumi{t}{m} \in \R[\enum{X}]$ such that $\R^{n} = \bigcup\limits_{k = 1}^{m} U_{k}$ and $t = t_{k}$ on $U_{k}$.

\label{ppf}
\end{definition}

This conjecture was proved in 1984 by Louis~Mah\'e in the case $n = 2$ (see \cite{Mah1}). For $n > 2$, it is still open.

In 1989, J.~J.~Madden formulated the Pierce-Birkhoff Conjecture for $\R^{n}$ in terms of the real spectrum of the polynomial ring $\R[\enum{X}]$. In doing so, he introduced the concept of a ``separating ideal''. He used separating ideals to define a property that makes sense for any commutative ring and he showed that this property is satisfied by $\R[\enum{X}]$ for $n=1,2,3,\ldots$ if and only if the Pierce-Birkhoff Conjecture is true. In 1995, D.~Alvis, B.~L.~Johnston and J.~J.~Madden applied Zariski's theory of quadratic transformations to study separating ideals in regular local domains, showing that quadratic transforms of separating ideas are well-behaved if the residue field is real closed. In 2007, F.~Lucas, J.~J.~Madden, D.~Schaub and M.~Spivakovsky introduced the Connectedness Conjecture and showed that it implies the Pierce-Birkhoff Conjecture.

In the present paper, we give a proof of the Connectedness Conjecture for the coordinate ring of any non-singular two-dimensional affine real algebraic variety over a real closed field. Of course, this implies the Pierce-Birkhoff Conjecture for such rings. Our proof rests on the theory developed by Madden and by Alvis, Johnston and Madden. In an attempt to make the present paper self-contained, we include a review the essential results of this theory.

The author's work is supported by the DFG-Graduiertenkolleg ``Mathemati\-sche Logik und Anwendungen'' (GRK 806) in Freiburg (Germany). He thanks Prof. Dr. Alexander~Prestel from the University of Konstanz for supervising his work and the referee for useful comments.

\section{Pierce-Birkhoff Rings}

In the next two sections, we will give a short overview of Madden's article \cite{Madd1} and we include a proof of one of the main results.

Let $A$ be a (commutative) ring (with $1$). Denote by $\Sper{A}$ the real spectrum of $A$. Let $\alpha \in \Sper{A}$. Then $\alpha$ induces a total ordering on $A(\alpha) := A / \supp{\alpha}$, where $\supp{\alpha} = \alpha \cap -\alpha$. We denote by $\rho_{\alpha}$ the homomorphism $A \rightarrow A / \supp{\alpha}$ and by $R(\alpha)$ the real closure of the quotient field $k(\alpha)$ of $A(\alpha)$ with respect to the ordering induced by $\alpha$.

\begin{definition}{(Schwartz)}

An element $s \in \prod_{\alpha \in \Sper{A}} R(\alpha)$ is called an \textbf{abstract semialgebraic function} on $\Sper{A}$ iff the image of $s$ is constructible, i.e., there exists a formula $\Phi(T)$ in the language of ordered rings with coefficients in $A$ containing only one variable, which is also free, such that, for all $\alpha \in \Sper{A}$, $\Phi_{\alpha}(s(\alpha))$ holds in $R(\alpha)$ and $s(\alpha)$ is the only solution of the specialization $\Phi_{\alpha}(T)$ in $R(\alpha)$.  We say that $s$ is \textbf{continuous} if it  satisfies the following compatibility condition regarding specializations:\\
		Let $\alpha, \beta \in \Sper{A}$, $\beta$ a specialization of $\alpha$. Then $R(\alpha)$ contains a largest convex subring $C_{\beta \alpha}$ with maximal ideal $M_{\beta \alpha}$ such that $A(\alpha) \subset C_{\beta \alpha}$ and $\rho_{\alpha}^{-1}(M_{\beta \alpha}) = \supp{\beta}$. Then $C_{\beta \alpha} / M_{\beta \alpha}$ is a real closed field containing $A(\beta)$ and therefore also $R(\beta)$. The condition $s$ has to suffice is \[ \lambda_{\beta \alpha}(s(\alpha)) = s(\beta) \in R(\beta) \subset C_{\beta \alpha} / M_{\beta \alpha}, \] where $\lambda_{\beta \alpha}\colon C_{\beta \alpha} \rightarrow C_{\beta \alpha} / M_{\beta \alpha}$.
	
We denote by $\SA{A}$ the set of all continuous abstract semialgebraic functions on $\Sper{A}$. This is a subring of $\prod_{\alpha \in \Sper{A}} R(\alpha)$.

\end{definition}

\begin{remarks}$ $

	\begin{enumerate}
	
		\item[(i)]{Every element $a \in A$ induces an abstract semialgebraic function.}
		\item[(ii)]{For $s,t \in \SA{A}$, the set $\set{\alpha \in \Sper{A}}{s(\alpha) = t(\alpha)} \subset \Sper{A}$ is constructible.}
	
	\end{enumerate}

\end{remarks}

\begin{definition}

	By \textbf{$\PW{A}$}, we denote the set of all continuous abstract semialgebraic functions $s \in \SA{A}$ with the property that, for all $\alpha \in \Sper{A}$, there exists an element $a \in A$ such that $s(\alpha) = a(\alpha)$.
	
	The elements of $\PW{\R[\enum{X}]}$ are called \textbf{abstract piecewise polynomial functions}.

\end{definition}

An abstract piecewise polynomial function has a presentation analogous to that of a piecewise polynomial as in Definition \ref{ppf}. To be precise, suppose $t \in \PW{A}$.  Let $\alpha \in \Sper{A}$. If $t(\alpha) = a(\alpha)$ for some $a \in A$, then, $t = a$ on a constructible set $U_{\alpha}$ containing $\alpha$. By compactness, we have $\Sper{A} = \bigcup\limits_{i = 1}^{m} U_{i}$ for finitely many constructible sets $U_{i}$ such that $t = a_{i}$ on $U_{i}$ for some $a_{i} \in A$. Since $\set{\alpha \in \Sper{A}}{t(\alpha) = a_{i}(\alpha)}$ is closed in the spectral topology, we may assume that each $U_{i}$ is closed.

\begin{definition}{(Madden)}

	Suppose $t \in \SA{A}$. We say $t$ is \textbf{sup-inf-definable over $A$} if there is a finite family $\{ h_{ij} \}_{i \in I,j \in J} \subset A$ such that
	\[ t = \sup_{i \in I} \big{(} \inf_{j \in J} h_{ij} \big{)}. \] 
	$A$ is called a \textbf{Pierce-Birkhoff ring} if every $t \in \PW{A}$ is sup-inf-definable over $A$.

\label{PBr}
\end{definition}

Let $R$ be a real closed field, and let $V$ be an algebraic subset of $R^{n}$. We denote by $\mathcal{P}(V)$ its coordinate ring. Then $\PW{\mathcal{P}(V)}$ is isomorphic to the ring of piecewise polynomial functions on $V$. Hence the question whether the Pierce-Birkhoff Conjecture holds for $V$ is equivalent to the question whether $\mathcal{P}(V)$ is a Pierce-Birkhoff ring.

\section{Separating Ideals}

The present section summarizes the main results of \cite{Madd1}, including the definition of separating ideals and their use in describing the abstract Pierce-Birkhoff property. First, using the abstract Pierce-Birkhoff property (Definition \ref{PBr}) together with the compactness of the constructible topology of the real spectrum, we show that an abstract piecewise polynomial is globally sup-inf-definable if it is sup-inf-definable on every pair of elements of the real spectrum. This shows that the Pierce-Birkhoff property is local in a very strong sense.

\begin{theorem}

	$A$ is a Pierce-Birkhoff ring if and only if, for all $t \in \PW{A}$ and for all $\alpha,\beta \in \Sper{A}$, there is an element $h \in A$ such that $h(\alpha) \ge t(\alpha)$ and $h(\beta) \le t(\beta)$.

\end{theorem}

\begin{proof}

	Let $t \in \PW{A}$.
	
	Suppose there is a finite family $\{ h_{ij} \}_{i \in I,j \in J} \subset A$ such that
	\[ t = \sup_{i \in I} \big{(} \inf_{j \in J} h_{ij} \big{)}, \]
	and suppose further that there exist $\alpha, \beta \in \Sper{A}$ such that, for all $h \in A$, $h(\alpha) \ge t(\alpha)$ implies $h(\beta) > t(\beta)$. Then, since there is some $i_{0} \in I$ such that $\inf_{j \in J} h_{i_{0}j}(\alpha) \ge t(\alpha)$, we have $\sup_{i \in I} \big{(} \inf_{j \in J} h_{ij}(\beta) \big{)} > t(\beta)$, a contradiction.
	
	Suppose now that, for all $\alpha,\beta \in \Sper{A}$, there is an element $h_{\alpha \beta} \in A$ such that $h_{\alpha \beta}(\alpha) \ge t(\alpha)$ and $h_{\alpha \beta}(\beta) \le t(\beta)$. In particular, $h_{\alpha \alpha}(\alpha) = t(\alpha)$.
	
	For each $\alpha, \beta \in \Sper{A}$, there are closed constructible sets $U(\alpha,\beta)$ and $V(\alpha,\beta)$ such that $\alpha \in U(\alpha,\beta)$, $\beta \in V(\alpha,\beta)$ , $h_{\alpha \beta} \ge t$ on $U(\alpha,\beta)$ and $h_{\alpha \beta} \le t$ on $V(\alpha,\beta)$. Without loss of generality $U(\alpha,\alpha) = V(\alpha,\alpha)$.
	
	We fix some $\alpha \in \Sper{A}$. Since $\Sper{A}$ is the union of the sets $V(\alpha,\beta)$ with $\beta \in \Sper{A}$, by compactness, there exist $\alpha = \beta_{0}, \ldots, \beta_{r} \in \Sper{A}$ such that $\Sper{A} = \bigcup\limits_{j = 0}^{r} V(\alpha,\beta_{j})$. Let $U(\alpha) := \bigcap\limits_{j = 0}^{r} U(\alpha,\beta_{j})$ and $H_{\alpha} := \inf_{j \in J} h_{\alpha,\beta_{j}}$, where $J = \{ 0, \ldots, r \}$. Then $H_{\alpha} \le t$ globally and $H_{\alpha} = t$ on $U(\alpha)$.
	
	$\Sper{A}$ is the union of the sets $U(\alpha)$ with $\alpha \in \Sper{A}$, and hence, again by compactness, there exist $\alpha_{1}, \ldots, \alpha_{s} \in \Sper{A}$ such that $\Sper{A} = \bigcup\limits_{i = 0}^{s} U(\alpha_{i})$. Then $t = \sup_{i \in I} H_{\alpha_{i}}$, where $I = \{ 1, \ldots, s \}$.

\end{proof}

\begin{definition}

	Let $A$ be a ring. For $\alpha,\beta \in \Sper{A}$, we denote by $\bqf{\alpha}{\beta}$ the ideal of $A$ generated by all $a \in A$ with the property $a(\alpha) \ge 0$ and $a(\beta) \le 0$. We will call $\bqf{\alpha}{\beta}$ the \textbf{separating ideal} of $\alpha$ and $\beta$.

\end{definition}

\begin{remarks}

	Let $\alpha,\beta \in \Sper{A}$.
	
	\begin{enumerate}
	
		\item{$\supp{\alpha} + \supp{\beta} \subset \bqf{\alpha}{\beta}$.}
		\item{In general, $\bqf{\alpha}{\beta}$ is not a prime ideal.}
	
	\end{enumerate}

\end{remarks}

\begin{lemma}

	Let $\alpha,\beta \in \Sper{A}$. Let $a \in A$ such that $a(\alpha) \ge 0$. Then we have $a \in \bqf{\alpha}{\beta}$ if and only if there exists an $h \in A$ such that $a(\alpha) \le h(\alpha)$ and $h(\beta) \le 0$.

\end{lemma}

\begin{definition}

	Let $\alpha \in \Sper{A}$. An ideal $I$ of $A$ is called \textbf{$\alpha$-convex} if, for all $a,b \in \alpha$, it follows from $a + b \in I$ that $a,b \in I$. The set of $\alpha$-ideals of $A$ is totally ordered by inclusion. If $A$ is noetherian, there exists a largest proper $\alpha$-convex ideal in $A$, called the \textbf{center} of $\alpha$ in $A$, and denoted by $\cent{\alpha}$.

\end{definition}

\begin{proposition}

	Let $\alpha,\beta \in \Sper{A}$.

	\begin{enumerate}
	
		\item[a)]{In $A$, the ideal $\bqf{\alpha}{\beta}$ is convex with respect to $\alpha$ and $\beta$.}
		\item[b)]{Both $\alpha$ and $\beta$ induce the same total ordering on $A / \bqf{\alpha}{\beta}$, and $\bqf{\alpha}{\beta}$ is the smallest ideal of $A$ with this property.}
		\item[c)]{If $\sqrt{\bqf{\alpha}{\beta}}$ is proper, then it is prime, and $\alpha$ and $\beta$ induce the same total ordering on $A / \sqrt{\bqf{\alpha}{\beta}}$. Thus, in this case, $\sqrt{\bqf{\alpha}{\beta}}$ together with this order is the least common specialization $\gamma$ of $\alpha$ and $\beta$ in $\Sper{A}$.}
		\item[d)]{Let $t \in \PW{A}$. For $\delta \in \Sper{A}$, we denote by $t_{\delta}$ any element $a \in A$ such that $t(\delta) = a(\delta)$. The compatibility condition for $t$ gives us $t_{\alpha}(\gamma) = t_{\gamma}(\gamma) = t_{\beta}(\gamma)$, hence $t_{\alpha} - t_{\beta} \in \sqrt{\bqf{\alpha}{\beta}}$.}
		\item[e)]{Every ideal of $A$ containing $\bqf{\alpha}{\beta}$ is $\alpha$-convex if and only if it is $\beta$-convex.}
		\item[f)]{Suppose $\alpha$ and $\beta$ have no common specialization. Then $\bqf{\alpha}{\beta} = A$.}
		\item[g)]{Suppose $A$ is noetherian and $\cent{\alpha} \neq \cent{\beta}$. Then $\bqf{\alpha}{\beta} = A$, since otherwise both $\cent{\alpha}$ and $\cent{\beta}$ would be $\alpha$- and $\beta$-convex, and therefore equal because of their maximality.}
	
	\end{enumerate}

\end{proposition}

\begin{theorem}

	$A$ is a Pierce-Birkhoff ring if and only if for all $t \in \PW{A}$ and all $\alpha,\beta \in \Sper{A}$, we have $t_{\alpha} - t_{\beta} \in \bqf{\alpha}{\beta}$.

\label{PBRCond}
\end{theorem}

\begin{corollary}

	Every field is Pierce-Birkhoff.

\end{corollary}

\begin{remarks}

	Let $\alpha, \beta \in \Sper{A}$.
	
	\begin{enumerate}
	
		\item[(i)]{If $\bqf{\alpha}{\beta}$ is a prime ideal or equal to $A$, then for each $t \in \PW{A}$, we have $t_{\alpha} - t_{\beta} \in \bqf{\alpha}{\beta}$.}
		\item[(ii)]{If $\alpha$ and $\beta$ have no common specialization, then $\bqf{\alpha}{\beta} = A$, so we have to check the condition of the theorem only for $\alpha$ and $\beta$ having a common specialization.}
		\item[(iii)]{If $A$ is a ring with the property that the localization at any real prime ideal is a discrete valuation ring, then $A$ is a Pierce-Birkhoff ring. See, for example, Lemma \ref{CCone} below.}
		\item[(iv)]{Any Dedekind ring is a Pierce-Birkhoff ring.}
		\item[(v)]{The (real) coordinate ring of any non-singular (real) algebraic curve is a Pierce-Birkhoff ring.}
	
	\end{enumerate}

\end{remarks}

\section{Connectedness}

The Connectedness Conjecture was introduced by Lucas, Madden, Schaub and Spivakovsky in \cite{LMSS} who showed that it implies the Pierce-Birkhoff Conjecture. We will review this work. First, in order to state the conjecture, we make the following definition.

\begin{definition}

	Let $A$ be a ring and let $\alpha, \beta \in \Sper{A}$.  We say $\alpha$ and $\beta$ satisfy the \textbf{connectedness condition} if for any $\enumi{g}{s} \in A \setminus \bqf{\alpha}{\beta}$, there exists a connected set $C \subset \Sper{A}$ such that $\alpha, \beta \in C$ and $C \cap \set{\delta \in \Sper{A}}{g_{j}(\delta) = 0} = \emptyset$ for all $j \in \{ 1, \ldots, s \}$.

\end{definition}

\begin{conjecture}{(Connectedness Conjecture)}

	Suppose $R$ is a real closed field and $A := R[\enum{X}]$.  Then every pair $\alpha, \beta \in \Sper{A}$ satisfies the connectedness condition.

\end{conjecture}

\begin{theorem}{(Lucas, Madden, Schaub, Spivakovsky)}

	Let $A$ be a noetherian ring in which the connectedness condition holds for every two points $\alpha,\beta \in \Sper{A}$ with a common center. Then $A$ is a Pierce-Birkhoff ring.

\end{theorem}

\begin{proof}

	Let $t \in \PW{A}$, and let $(U_{j})_{j=1}^{m}$ be a finite sequence of constructible sets in $\Sper{A}$ such that $\Sper{A} = \bigcup\limits_{j=1}^{m} U_{j}$ and $t = t_{j}$ on $U_{j}$ for some $t_{j} \in A$. Let $\alpha,\beta \in \Sper{A}$. By Theorem \ref{PBRCond}, we have to show that $t_{\alpha} - t_{\beta} \in \bqf{\alpha}{\beta}$, where $t_{\alpha}$ (resp. $t_{\beta}$) is any element $a \in A$ such that $t(\alpha) = a(\alpha)$ (resp. $t(\beta) = a(\beta)$). We may assume that $\alpha$ and $\beta$ have a common center, since otherwise $A = \bqf{\alpha}{\beta}$. Now let $T = \set{\{ j,k \} \subset \{ 1, \ldots, m \}}{t_{j} - t_{k} \notin \bqf{\alpha}{\beta}}$, and we apply the connectedness condition to the finitely many elements $t_{j} - t_{k}$ with $\{ j,k \} \in T$ to get a connected set $C \subset \Sper{A}$ such that $\alpha,\beta \in C$ and $C \cap \set{\delta \in \Sper{A}}{(t_{j} - t_{k})(\delta) = 0} = \emptyset$ for all $\{ j,k \} \in T$.
	
	Let $K$ be the set of all indices $k \in \{ 1, \ldots, m \}$ such that there exists a sequence $\enumi{j}{s} \in \{ 1, \ldots, m \}$ with $\alpha \in U_{j_{1}}$, $j_{s} = k$, and, for all $q \in \{ 1, \ldots, s-1 \}$, we have $C \cap \set{\delta \in \Sper{A}}{(t_{j_{q}} - t_{j_{q+1}})(\delta) = 0} \neq \emptyset$.\\
	Let $F = \bigcup\limits_{k \in K} (U_{k} \cap C)$. Then $\alpha \in F$ by definition.
	
	We claim that $F = C$. Let $K^{c} := \{ 1,\ldots, m \} \setminus K$ and $G := \bigcup\limits_{j \in K^{(c)}} (U_{j} \cap C)$. Clearly, $C = F \cup G$ and both sets $F$ and $G$ are closed in $C$. Suppose $F \cap G \neq \emptyset$, and let $\delta \in F \cap G$. Then there exists some $k \in K$ and some $j \in K^{c}$ such that $\delta \in U_{k} \cap U_{j}$, and thus $t_{k}(\delta) = t_{j}(\delta)$. Therefore, we have $\delta \in C \cap \{ t_{k} - t_{j} = 0 \}$, and hence $j \in K$, a contradiction. We have shown that $F$ and $G$ are disjoint, and since $C$ is connected and $F \neq \emptyset$, this yields $G = \emptyset$. In particular, we have $\beta \in F$.
	
	Now let $k \in K$ such that $\beta \in U_{k}$, and hence we can set $t_{k} =: t_{\beta}$.
	Then there exists a sequence $\enumi{j}{s} \in \{ 1, \ldots, m \}$ such that $\alpha \in U_{j_{1}}$, $j_{s} = k$ and, for all $q \in \{ 1, \ldots, s-1 \}$, $\{ j_{q},j_{q+1} \} \notin T$, i.e., $t_{j_{q}} - t_{j_{q+1}} \in \bqf{\alpha}{\beta}$. Hence, we have obtained $t_{\alpha} - t_{\beta} \in \bqf{\alpha}{\beta}$, if we set $t_{\alpha} := t_{j_{1}}$.

\end{proof}

If $\bqf{\alpha}{\beta} = \{ 0 \}$, then we have $\alpha = \beta$ and $\supp{\alpha} = \supp{\beta} = \{ 0 \}$. Hence $g(\alpha) \neq 0$ for all $g \in A \setminus \bqf{\alpha}{\beta}$, and therefore the set $C = \{ \alpha \}$ fulfills all requirements of the connectedness condition.

In the next sections, we will prove that the connectedness condition holds for every pair of points $\alpha,\beta$ which are in the real spectrum of a finitely generated two-dimensional regular $R$-algebra $A$, where $R$ is a real closed field, and have the same center in $A$, i.e., $\bqf{\alpha}{\beta} \subsetneq A$. Note that, if $\bqf{\alpha}{\beta} = A$, the connectedness condition holds for $\alpha$ and $\beta$ if and only if there exists a connected set $C \subset \Sper{A}$ that contains $\alpha$ and $\beta$. This is true, for example, if $A$ is the polynomial ring $R[\enum{X}]$.

First, we give some examples for connected sets in the real spectrum of a polynomial ring over a real closed field.  Let $R$ be a real closed field, and let $A := R[\enum{X}]$ be the polynomial ring in $n$ indeterminates over $R$.

\begin{definition}

	A semialgebraic subset $S$ of $R^{n}$ is called \textbf{semialgebraically connected} if there are no two non-empty closed semialgebraic sets $S_{1}$ and $S_{2}$ in $S$ such that $S_{1} \cap S_{2} = \emptyset$ and $S = S_{1} \cup S_{2}$.

\end{definition}

\begin{proposition}

	Let $U,V$ be two semialgebraic sets, and let $\varphi\colon U \rightarrow V$ be a continuous semialgebraic map, i.e., a continuous map whose graph is semialgebraic. Then the image under $\varphi$ of any semialgebraically connected subset of $U$ is again semialgebraically connected.

\label{ContSemiConn}
\end{proposition}

\begin{proposition}

	Every interval of $R$ is semialgebraically connected.

\label{IvSemiConn}
\end{proposition}

From Proposition \ref{IvSemiConn}, one can immediately derive the following result.

\begin{proposition}

	Every convex semialgebraic set $S \subset R^{n}$ is semialgebraically connected.

\label{ConvSemiConn}
\end{proposition}

\begin{examples}

	Let $\varepsilon \in R$ such that $\varepsilon > 0$, and let $I \subset \{ 1, \ldots, n \}$. Then, by Proposition \ref{ConvSemiConn}, the semialgebraic set
	\[ C_{I}^{\varepsilon} := \set{x \in R^{n}}{\varepsilon - \sum_{i = 1}^{n} x_{i}^{2} > 0, \ x_{i} > 0 \ (i \in I)} \]
	is semialgebraically connected.

\label{ExSemiConn}
\end{examples}

Let $V \subset R^{n}$ be an algebraic set, and let $\mathcal{P}(V)$ be its coordinate ring. For any semialgebraic set $S$ in $V$, we denote by $\tilde{S}$ the corresponding constructible set in $\Sper{\mathcal{P}(V)}$. Let us recall Proposition 7.5.1 of \cite{BCR}, which is an important tool to find connected sets in the real spectrum of $\mathcal{P}(V)$.

\begin{proposition}

	Let $S$ be a semialgebraic set in $V$. $S$ is semialgebraically connected if and only if $\tilde{S}$ is connected in the spectral topology of $\mathcal{P}(V)$.

\label{SemiConnSpec}
\end{proposition}

\section{The One-Dimensional Case}

Let $A$ be an integral domain. In this section, we treat the case where $\sqrt{\bqf{\alpha}{\beta}}$ is a prime ideal of height one, and the localization of $A$ at $\sqrt{\bqf{\alpha}{\beta}}$ is a regular ring.

\begin{lemma}

	 Let $A$ be a integral domain. Let $\alpha, \beta \in \Sper{A}$ such that $\sqrt{\bqf{\alpha}{\beta}}$ is a prime ideal of height one and $A_{\sqrt{\bqf{\alpha}{\beta}}}$ is a regular local ring. Let $\enumi{g}{s} \in A \setminus \bqf{\alpha}{\beta}$. Then $\sqrt{\bqf{\alpha}{\beta}} = \bqf{\alpha}{\beta}$ and there exists a connected set $C \subset \Sper{A}$ such that $\alpha, \beta \in C$ and $C \cap \set{\delta \in \Sper{A}}{g_{j}(\delta) = 0} = \emptyset$ for all $j \in \{ 1,\ldots,s \}$.

\label{CCone}
\end{lemma}

\begin{proof}

	Assume $\bqf{\alpha}{\beta} \subsetneq \sqrt{\bqf{\alpha}{\beta}}$. We consider the one-dimensional regular local ring (i.e., discrete valuation ring) $B := A_{\sqrt{\bqf{\alpha}{\beta}}}$. Both orderings $\alpha,\beta \in \Sper{A}$ extend uniquely to $\alpha',\beta' \in \Sper{B}$. In $B$, the separating ideal $\bqf{\alpha'}{\beta'}$ is equal to $\bqf{\alpha}{\beta}B$. By assumption, there exists some $\pi \in \sqrt{\bqf{\alpha}{\beta}} \setminus \bqf{\alpha}{\beta}$ such that $\pi B = \sqrt{\bqf{\alpha}{\beta}}B = \sqrt{\bqf{\alpha'}{\beta'}}$.\\
	Every element in $b \in B$ can be written as $\pi^{r}u$ for some $r \in \N$ and $u \in B^{\times} = B \setminus \sqrt{\bqf{\alpha'}{\beta'}}$. Both $\pi$ and $u$ do not change sign between $\alpha'$ and $\beta'$, hence $b = \pi^{r}u$ does not change sign either. Thus $\sqrt{\bqf{\alpha'}{\beta'}} = \bqf{\alpha'}{\beta'} = \{ 0 \}$, a contradiction.
	
	Let $\gamma \in \Sper{A}$ be the least common specialization of $\alpha$ and $\beta$. Then $\supp{\gamma} = \sqrt{\bqf{\alpha}{\beta}} = \bqf{\alpha}{\beta}$. Hence, for all $g \in A \setminus \bqf{\alpha}{\beta}$, we have $\gamma \in \Sper{A} \setminus \set{\delta \in \Sper{A}}{g(\delta) = 0}$. For any such $g$, let $C_{g,\gamma}$ be the connected component of the open set $\Sper{A} \setminus \set{\delta \in \Sper{A}}{g(\delta) = 0}$ that contains $\gamma$. Since $\alpha$ and $\beta$ specialize to $\gamma$, they are also contained in $C_{g,\gamma}$.\\
	Now let $g := g_{1} \cdots g_{s}$. Then $g \notin \bqf{\alpha}{\beta} = \sqrt{\bqf{\alpha}{\beta}}$, so we can take $C := C_{g,\gamma}$.

\end{proof}

\section{Valuations, Orderings and Quadratic Transformations}

Let $R$ be a real closed field. In order to prove the connectedness condition in the case that $\sqrt{\bqf{\alpha}{\beta}}$ has height two in a finitely generated two-dimensional regular $R$-algebra, we have to consider so-called quadratic transformations along a valuation of this ring.

Let $A$ be a noetherian ring.

\begin{definition}

	A \textbf{valuation} $v$ of $A$ is a map $A \rightarrow \Gamma \cup \{ \infty \}$, where $\Gamma$ is a total ordered abelian group, such that for all $a,b \in A$
	
	\begin{enumerate}
	
		\item[(i)]{$v(0) = \infty$, $v(1) = 0$,}
		\item[(ii)]{$v(ab) = v(a) + v(b)$ and}
		\item[(iii)]{$v(a+b) \ge \min \{ v(a), v(b) \}$.}
	
	\end{enumerate}
	
	The prime ideal $\supp{v} := \set{a \in A}{v(a) = \infty}$ is called the \textbf{support} of $v$ in $A$ and, if $v$ is non-negative on $A$, then $\cent{v} := \set{a \in A}{v(a) > 0}$ is a prime ideal called the \textbf{center} of $v$ in $A$.
	
	Let $v$ be a valuation of $A$ that is non-negative on $A$.
	
	An ideal $I$ of $A$ is called a \textbf{$v$-ideal} if $I = \set{a \in A}{\exists \, b \in I \ (v(b) \le v(a))}$. Since $A$ is noetherian, for all $v$-ideals $I$, there exists an element $b \in I$ such that $I = \set{a \in A}{v(b) \le v(a)}$. If $I$ is a $v$-ideal, then $I^{v} := \set{a \in A}{v(b) < v(a)}$ is the largest $v$-ideal properly contained in $I$. The set of $v$-ideals of $A$ is totally ordered by inclusion.
	
	Note that the support of $v$ is the smallest $v$-ideal in $A$ and that the center of $v$ is the largest proper $v$-ideal in $A$.
	
	If $A$ is a local ring with maximal ideal $\m$, we say that a valuation $v$ of $A$ \textbf{dominates} $A$ if $v$ is non-negative on $A$ and the center of $v$ in $A$ is equal to $\m$.

\end{definition}

\begin{remark}

	Let $A$ be a local ring with maximal ideal $\m$ and residue field $k$, and let $v$ be a valuation of $A$ that dominates $A$. Then $v$ induces a valuation on the quotient field of $A / \supp{v}$. We denote the corresponding valuation ring by $\OR_{v}$, its maximal ideal by $\m_{v}$ and its residue field by $K_{v}$. Since $v$ dominates $A$, we have $k \subset K_{v}$.
	
	Let $I$ be a $v$-ideal of $A$ where $I$ is different from the support of $v$. Consider the following composition of $k$-vector space homomorphisms
	\[ I \twoheadrightarrow I / \supp{v} \rightarrow \OR_{v} / \m_{v} = K_{v}, \ a \mapsto \overline{a} = a + \supp{v} \mapsto \frac{\overline{a}}{\overline{b}} + \m_{v},  \]
	where $b$ is an element of $I$ having minimal value in $I$. The kernel of this composition is clearly $I^{v}$, hence $I / I^{v}$ is a sub-$k$-vector space of $K_{v}$.

\label{viq1}
\end{remark}

We will now assign to each ordering $\alpha \in \Sper{A}$ of $A$ a valuation $v_{\alpha}$.

\begin{definition}

	Let $\alpha \in \Sper{A}$. Then $\alpha$ induces a total ordering on the field $k(\alpha) = \Quot{A/\supp{\alpha}}$. Now let $\OR_{\alpha}$ be the convex hull of $A(\alpha)$ in $k(\alpha)$. This is a valuation ring of $k(\alpha)$. Let $v'_{\alpha}$ be a corresponding valuation, then $v_{\alpha} := v'_{\alpha} \circ \rho_{\alpha}$ is a valuation of $A$. For any $v_{\alpha}$-ideal $I$ of $A$, we write $I^{\alpha}$ instead of $I^{v_{\alpha}}$.

\end{definition}

\begin{remark}

	Let $\alpha, \beta \in \Sper{A}$. An ideal $I$ of $A$ is a $v_{\alpha}$-ideal if and only if it is convex with respect to $\alpha$. Hence $\bqf{\alpha}{\beta}$ is a $v_{\alpha}$-ideal.

\end{remark}

From now on, let $A$ be an regular local domain with maximal ideal $\m$ and residue field $k$. In particular, $A$ is integrally closed. Let $\ord{A}$ be the order valuation of $A$ (i.e., $\ord{A}(a) = \max \set{n \in \N}{a \in \m^{n}}$). In \cite{ZS2} (Appendix 5) Oscar Zariski and Pierre Samuel showed a unique factorization theorem for $v$-ideals in a two-dimensional regular local ring where $v$ is a dominating valuation.

\begin{definition}

	An ideal $I$ of $A$ is called \textbf{simple} if it is proper and it cannot be written as a product of proper ideals.

\end{definition}

\begin{remark}

	The $\ord{A}$-ideals of $A$ are exactly the powers of the maximal ideal $\m$, hence $\m$ is the only simple $\ord{A}$-ideal of $A$.

\end{remark}

\begin{theorem}

	If $A$ has dimension two and $v$ is a valuation of $A$ that dominates $A$, then a $v$-ideal of $A$ is simple if and only if it cannot be written as a product of proper $v$-ideals of $A$. Moreover, every $v$-ideal of $A$ different from $(0)$ and $A$ has a unique factorization into simple $v$-ideals.

\label{ZarUniFac}
\end{theorem}

\begin{remark}

	Actually, Zariski and Samuel proved this unique factorization theorem for complete ideals (\cite{ZS2}, Appendix 5, Theorem 3). In an integrally closed domain $A$, an ideal $I$ is said to be \textbf{complete} if it is integrally closed in $\Quot{A}$, i.e., for all $a \in \Quot{A}$ such that $a^{m} + b_{1}a^{m-1} + \cdots + b_{m} = 0$, where $b_{j} \in I^{j}$, we already have $a \in I$.
	But if $v$ is a valuation of $A$ that is non-negative on $A$, then every $v$-ideal $I$ of $A$ is complete, and if a $v$-ideal $I$ is a product of simple complete ideals, then every factor is already a $v$-ideal.

\label{vic}
\end{remark}

In \cite{AJM}, Alvis, Johnston and Madden give a sufficient condition for the simplicity of the separating ideal of two points in the real spectrum that are centered at the same maximal ideal.

\begin{proposition}

	Suppose $\alpha,\beta \in \Sper{A}$ are both centered at the maximal ideal $\m$ of $A$. Let $k := A / \m$. If $I / I^{\alpha} \cong k$ for all $v_{\alpha}$-ideals $I$ which properly contain $\bqf{\alpha}{\beta}$, then $\bqf{\alpha}{\beta}$ is simple.

\label{Simp}
\end{proposition}

\begin{proof}

	Let $x \in \bqf{\alpha}{\beta}$ such that
	\begin{enumerate}
	
		\item[(i)]{$x(\alpha) \ge 0$ and $x(\beta) \le 0$,}
		\item[(ii)]{$x$ has minimal $v_{\alpha}$-value in $\bqf{\alpha}{\beta}$, and}
		\item[(iii)]{$x$ has minimal $v_{\beta}$-value in $\bqf{\alpha}{\beta}$.}
	
	\end{enumerate}
	Note that such an element always exists, since there must be an element $x_{\alpha}$ that satisfies (i) and (ii) and there must be an element $x_{\beta}$ that satisfies (i) and (iii), and if $x_{\alpha}$ does not satisfy (iii) and $x_{\beta}$ does not satisfy (ii), then $x_{\alpha} + x_{\beta}$ satisfies all three conditions.
	
	Suppose $\bqf{\alpha}{\beta}$ is not simple. Then, by Theorem \ref{ZarUniFac}, it can be written as the product of two proper $v_{\alpha}$-ideals $I$ and $J$. Since they (properly) contain $\bqf{\alpha}{\beta}$, they are also $v_{\beta}$-ideals. Now we can write $x = \sum_{k = 1}^{r} a_{k}b_{k}$, where, for all $k \in \{ 1, \ldots, r \}$, $a_{k} \in I$ and $b_{k} \in J$. Note that if $a_{k} \in I^{\alpha} = I^{\beta}$ or $b_{k} \in J^{\alpha} = J^{\beta}$, then $a_{k}b_{k} \in \bqf{\alpha}{\beta}^{\alpha} \cap \bqf{\alpha}{\beta}^{\beta}$. Since $x$ satisfies (ii) and (iii), without loss of generality we can write $x = \sum_{k = 1}^{s} a_{k}b_{k} + c$, where $s \in \{ 1, \ldots r \}$, $a_{k} \in I \setminus I^{\alpha}$ and $b_{k} \in J \setminus J^{\alpha}$ for all $k \le s$, and $c \in \bqf{\alpha}{\beta}^{\alpha} \cap \bqf{\alpha}{\beta}^{\beta}$. Let $a \in I \setminus I^{\alpha}$ and $b \in J \setminus J^{\alpha}$. Since $I / I^{\alpha} \cong k \cong J / J^{\alpha}$, we have, for $k \le s$, that $a_{k} = u_{k}a + a_{k}'$  and $b_{k} = v_{k}b + b_{k}'$ for some $u_{k},v_{k} \in A^{\times}$, $a_{k}' \in I^{\alpha}$ and $b_{k}' \in  J^{\alpha}$. Then $x = ab \sum_{k = 1}^{s} u_{k}v_{k} + c'$, where $v_{\alpha}(c') > v_{\alpha}(\bqf{\alpha}{\beta}) = v_{\alpha}(x)$ and $v_{\beta}(c') > v_{\beta}(\bqf{\alpha}{\beta}) = v_{\beta}(x)$, and thus we can derive from (i) that $(x - c')(\alpha) \ge 0$ and $(x - c')(\beta) \le 0$. Since $x$ satisfies (ii) and (iii), we have $\sum_{k = 1}^{s} u_{k}v_{k} \in A^{\times}$. So $a$ or $b$ must change sign between $\alpha$ and $\beta$, but this is impossible, since they are not elements of $\bqf{\alpha}{\beta}$.

\end{proof}

\begin{definition}

	A \textbf{quadratic transform} of $A$ is a local ring $B = (A[x^{-1}\m])_{\p}$, where $\ord{A}(x) = 1$ and $\p$ is a prime ideal of $A[x^{-1}\m] = \set{\frac{a}{x^{m}}}{\ord{A}(a) \ge m}$ such that $\p \cap A = \m$.

\end{definition}

Under a suitable condition, one can extend orderings of $A$ to a quadratic transform of $A$ (\cite{AJM}).

\begin{lemma}

	Let $\alpha \in \Sper{A}$, and let $B = (A[x^{-1}\m])_{\p}$ be a quadratic transform of $A$. If $\supp{\alpha} \neq \m$, then there is a unique $\alpha' \in \Sper{B}$ such that $\alpha' \cap A = \alpha$.

\end{lemma}

\begin{definition}
	
	Let $v$ be a non-trivial valuation (i.e., $\cent{v} \neq \supp{v}$) of $A$ that dominates $A$. The \textbf{quadratic transform of $A$ along $v$} is defined to be the ring $B = S^{-1} A[x^{-1}\m]$, where $x$ is an element of $\m$ of minimal value and $S = \set{a \in A[x^{-1}\m]}{v(a) = 0}$. $B$ is again a regular local ring, independent of the choice of $x$, $v$ is extendable to $B$ and it dominates $B$.
	
	We can then iterate this process and derive a sequence of quadratic transformations along $v$ starting from $A$, denoted by
	\[ A = A^{(0)} \prec A^{(1)} \prec \cdots. \]
	This sequence may be infinite.

\end{definition}

\begin{definition}

	Let $B = (A[x^{-1}\m])_{\p}$ be a quadratic transform of $A$. Let $I$ be an ideal of $A$ with $\ord{A}(I) = r$. Then $\frac{a}{x^{r}} \in A[x^{-1}\m]$ for all $a \in I$. Hence $I A[x^{-1}\m] = x^{r}I'$ for some ideal $I'$ of $A[x^{-1}\m]$. The ideal $T(I) := I'B$ is called the \textbf{transform of $I$ in $B$}. Now let $J$ be an ideal of $B$. Since $J$ is finitely generated, there is a smallest integer $n \in \N$ such that $x^{n}J = W(J)B$ for some ideal $W(J)$ of $A$, called the \textbf{inverse transform of $J$}.

\end{definition}

\begin{remarks}

	We consider a quadratic transform $B$ of $A$.
	\begin{enumerate}
	
		\item{For all ideals $I$ of $A$ and all ideals $J$ of $B$, we always have $T(W(J)) = J$, but in general only $W(T(I)) \supset I$.}
		\item{The transformation of ideals is not order-preserving.}
		\item{For all ideals $I,J$ of $A$, we have $T(IJ) = T(I)T(J)$.}
	
	\end{enumerate}

\end{remarks}

Zariski and Samuel showed that in dimension two, any simple $\m$-primary complete ideal can be transformed into a maximal ideal by a suitable sequence of quadratic transformations (again see \cite{ZS2}, Appendix 5). Applying this result to simple $\m$-primary $v$-ideals yields the following.

\begin{theorem}

	Suppose the dimension of $A$ is two. Let $v$ be a non-trivial valuation of $A$ that dominates $A$ and is different from the order valuation $\ord{A}$. Let $A'$ be the quadratic transform of $A$ along $v$, and let $\m'$ be its maximal ideal. Let $\mathcal{S}$ be the set of all simple $\m$-primary $v$-ideals of $A$, and let $\mathcal{S'}$ be the set of all simple $\m'$-primary $v$-ideals of $A'$. Then $A'$ has again dimensional two, the residue field of $A'$ is an algebraic extension of the residue field of $A$, every transform of a $\m$-primary $v$-ideal is again a $v$-ideal, and we have that the sets $\mathcal{S} \setminus \{ \m \}$ and $\mathcal{S'}$ are in one-to-one correspondence via the order-preserving maps $\mathcal{I} \mapsto T(\mathcal{I})$ and $\mathcal{J} \mapsto W(\mathcal{J})$.\\
	Furthermore, for every simple $\m$-primary $v$-ideal $\mathcal{I}$, there exists some $s \in \N$ and a sequence $A = A^{(0)} \prec \cdots \prec A^{(s)}$ of quadratic transformations along $v$ such that the iterated transform $T^{(s)}(\mathcal{I})$ of $\mathcal{I}$ equals $\m^{(s)}$, the maximal ideal in $A^{(s)}$.

\label{ZarQuadTrans}
\end{theorem}

\begin{remark}

	Suppose the dimension of $A$ is two. Let $v$ be a non-trivial valuation of $A$ that dominates $A$. Consider the quadratic transform of $A$ along $v$. Let $I$ be a $\m$-primary $v$-ideal of $A$, and let $r = \ord{A}{I}$ be the order of $I$. By Theorem \ref{ZarQuadTrans}, the transform $T(I)$ is again a $v$-ideal, so we may consider the following composition of $k$-vector space homomorphisms
	\[ I \leftrightarrow x^{-r} I \hookrightarrow T(I) \twoheadrightarrow T(I) / T(I)^{v}, \ a \mapsto \frac{a}{x^{r}} \mapsto \frac{a}{x^{r}} \mapsto \frac{a}{x^{r}} + T(I)^{v}, \]
	where $x$ is an element of $\m$ having minimal value. The kernel of this composition is the ideal $I^{v}$, hence $I / I^{v} \subset T(I) / T(I)^{v}$.

\label{viq2}
\end{remark}

Although in general, the transformation of ideals is not order-preserving, using the Unique Factorization Theorem \ref{ZarUniFac}, Theorem \ref{ZarQuadTrans} and the multiplicativeness of the ideal transformation, the following can be observed.

\begin{lemma}

	Suppose the dimension of $A$ is two. Let $v$ be a non-trivial valuation of $A$ that dominates $A$ and is different from the order valuation $\ord{A}$, and let $I$ be a $v$-ideal such that $I$ properly contains a simple $\m$-primary $v$-ideal $\mathcal{I}$. Then $T^{(s)}(I) = A^{(s)}$, where $T^{(s)}$ denotes the iterated ideal transformation with respect to the sequence of quadratic transformations $A = A^{(0)} \prec \cdots \prec A^{(s)}$ of $A$ along $v$ with the property $T^{(s)}(\mathcal{I}) = \m^{(s)}$.

\label{IdPropTrans}
\end{lemma}

Now we assume that $A$ has dimension two and that the residue field of $A$ is real closed, and we consider quadratic transformations along a valuation corresponding to a point in the real spectrum of $A$. The next theorem is the main result of \cite{AJM} and important for the proof of the (two-dimensional) Connectedness Conjecture below.

\begin{theorem}{(Alvis, Johnston, Madden)}

	Let $A = (A,\m,R)$ be a two-dimensional regular local domain such that $R$ is real closed. Let $\alpha,\beta \in \Sper{A}$ such that $\cent{\alpha} = \m = \cent{\beta}$ and $\bqf{\alpha}{\beta} \subsetneq \m = \sqrt{\bqf{\alpha}{\beta}}$.\\
	Consider the quadratic transformation along $v_{\alpha}$. Then:
	\[ T(\bqf{\alpha}{\beta}) = \bqf{\alpha'}{\beta'} \ \textrm{and} \ W(\bqf{\alpha'}{\beta'}) = \bqf{\alpha}{\beta}, \]
	where $\alpha'$ and $\beta'$ are the unique extensions of $\alpha$ and $\beta$ with respect to the quadratic transform of $A$ along $v_{\alpha}$.\\	
	Furthermore, there exists an integer $r \in \N$ and a sequence $A = A^{(0)} \prec \cdots \prec A^{(r)}$ of quadratic transformations along $v_{\alpha}$ such that the iterated transform $T^{(r)}(\bqf{\alpha}{\beta})$ of $\bqf{\alpha}{\beta}$ equals $\m^{(r)}$, the maximal ideal in $A^{(r)}$.

\label{AJMQuadTrans}
\end{theorem}

\begin{proof}

	Let $\alpha,\beta \in \Sper{A}$ such that $\cent{\alpha} = \cent{\beta} = \m$ and $\bqf{\alpha}{\beta}$ is $\m$-primary, but properly contained in $\m$.
	
	Let $I$ be a $v_{\alpha}$-ideal that properly contains $\bqf{\alpha}{\beta}$. Then $I$ properly contains a simple $\m$-primary $v_{\alpha}$-ideal:\\
	Since $\sqrt{\bqf{\alpha}{\beta}} = \m$, there are only finitely many $v_{\alpha}$-ideals bigger than $I$. In \cite{ZS2} (Appendix 5), it is shown that $v_{\alpha}$ is a prime divisor, i.e., its residue field has transcendence degree $1$ over $R$, if and only if there are only finitely many simple $\m$-primary $v_{\alpha}$-ideals. If $v_{\alpha}$ is not a prime divisor, then one of the infinitely many simple $\m$-primary $v_{\alpha}$-ideals must be properly contained in $I$.  If it is a prime divisor, then, according to \cite{AJM} (Theorem 4.4), $\bqf{\alpha}{\beta}$ contains a simple $\m$-primary $v_{\alpha}$-ideal, which is therefore properly contained in $I$.
	
	From the fact that $I$ properly contains a simple $\m$-primary $v_{\alpha}$-ideal, one concludes that $I/I^{\alpha} \cong R$: By Theorem \ref{ZarQuadTrans}, there is a sequence of quadratic transformations $A = A^{(0)} \prec \cdots A^{(s)}$ along $v_{\alpha}$ such that this simple $\m$-primary $v_{\alpha}$-ideal is transformed into the maximal ideal $\m^{(s)}$ of $A^{(s)}$, and therefore $I$ is transformed into $A^{(s)}$ (Lemma \ref{IdPropTrans}). Since $A^{(s)} / \m^{(s)}$ is a real algebraic extension of $R$, they are equal. Thus, by Remark \ref{viq2}, we have $I / I^{\alpha} \cong R$.
	(Note that if $v_{\alpha}$ is not a prime divisor, then $K_{v_{\alpha}} = R$, and therefore $I / I^{\alpha} \cong R$ already follows from Remark \ref{viq1}.)
	
	By Proposition \ref{Simp}, we have that $\bqf{\alpha}{\beta}$ is simple. It is true in general that $T(\bqf{\alpha}{\beta}) \subset \bqf{\alpha'}{\beta'}$ (\cite{AJM}, Lemma 3.2), and with the considerations above one easily shows that $W(\bqf{\alpha'}{\beta'}) \subset \bqf{\alpha}{\beta}$ (\cite{AJM}, Lemma 4.7). Applying $W$ on the first inclusion and $T$ on the second, one gets, by Theorem \ref{ZarQuadTrans}, $\bqf{\alpha}{\beta} \subset W(\bqf{\alpha'}{\beta'})$ and $\bqf{\alpha'}{\beta'} \subset T(\bqf{\alpha}{\beta})$. Altogether, we have the desired equalities.
	
	Again using Theorem \ref{ZarQuadTrans}, the last assertion follows from the simplicity of $\bqf{\alpha}{\beta}$.

\end{proof}

\section{The Connectedness Conjecture for Smooth Affine Surfaces over Real Closed Fields}

In this last section, we shall prove the Connectedness Conjecture for the coordinate ring of a non-singular two-dimensional affine real algebraic variety over a real closed field. That is, we shall show that any two points in the real spectrum of such a coordinate ring which have the same center satisfy the connectedness condition. We can assume that the variety is irreducible. By Lemma \ref{CCone}, we then need to consider only points $\alpha,\beta$ where $\sqrt{\bqf{\alpha}{\beta}}$ has height two. We will use Theorem \ref{AJMQuadTrans} to simplify the problem of constructing a suitable connected set.

Let $A = (A,\m,R)$ be a two-dimensional regular local domain such that $R = A / \m$ is real closed. Let $\alpha,\beta \in \Sper{A}$ such that $\cent{\alpha} = \m = \cent{\beta}$ and $\bqf{\alpha}{\beta} \subsetneq \m = \sqrt{\bqf{\alpha}{\beta}}$. By Theorem \ref{AJMQuadTrans}, there exists a finite sequence $(A,\m,R) =: (A^{(0)},\m^{(0)},k^{(0)}) \prec \cdots \prec (A^{(r)},\m^{(r)},k^{(r)})$ of quadratic transformations along the valuation $v_{\alpha}$ such that $\bqf{\alpha^{(r)}}{\beta^{(r)}} = T^{(r)}(\bqf{\alpha}{\beta}) = \m^{(r)}$, where $\alpha^{(r)}$ and $\beta^{(r)}$ are the unique extensions of $\alpha$ and $\beta$ to $A^{(r)}$.

At first we take a closer look at these quadratic transformations. The quadratic transformation $(A^{(i)},\m^{(i)},k^{(i)}) \prec (A^{(i+1)},\m^{(i+1)},k^{(i+1)})$ is a transformation along $v_{\alpha^{(i)}}$, which is the unique extension of $v_{\alpha}$ to $A^{(i)}$. We will write $v$ instead of $v_{\alpha^{(i)}}$ and sometimes $\alpha$ and $\beta$ instead of $\alpha^{(i)}$ and $\beta^{(i)}$.

By Theorem \ref{ZarQuadTrans}, we have that, for all $i \in \{ 0,\ldots,r \}$, the regular local ring $A^{(i)}$ has dimension two and residue field $k^{(i)} = R$, since $k^{(i)}$ is a real field. Now let $(x_{i},y_{i})$ be a regular system of local parameters of $A^{(i)}$, i.e., $m^{(i)} = (x_{i},y_{i})$. From now on, we will assume that $0 < v(x_{i}) \le v(y_{i})$. Then, a regular system $(x_{i+1},y_{i+1})$ of parameters of $A^{(i+1)}$ such that $0 < v(x_{i+1}) \le v(y_{i+1})$ can be derived from $(x_{i},y_{i})$ in the following way (see \cite{ZS2}, Appendix 5, proof of Proposition 1):

\begin{enumerate}

	\item[I.]{Suppose $v(x_{i}) < v(y_{i})$:
		\begin{enumerate}
			\item[1.]{If $v(x_{i}) \le v(\frac{y_{i}}{x_{i}})$, then let  $x_{i+1} := x_{i}$ and $y_{i+1} := \frac{y_{i}}{x_{i}}$.}
			\item[2.]{If $v(x_{i}) > v(\frac{y_{i}}{x_{i}})$, then let $x_{i+1} := \frac{y_{i}}{x_{i}}$ and $y_{i+1} := x_{i}$.}
		\end{enumerate}}
	\item[II.]{Suppose $v(x_{i}) = v(y_{i})$. Pick an element $u \in {A^{(i)}}^{\times}$ such that $\frac{y_{i}}{x_{i}} - u$ has positive value. (This is possible because $v(\frac{y_{i}}{x_{i}}) = 0$, hence there exists some element $u \in {A^{(i)}}^{\times}$ such that $\overline{u}^{v} = \overline{\frac{y_{i}}{x_{i}}}^{v} \in R$.)
		\begin{enumerate}
			\item[1.] {If $v(x_{i}) \le v(\frac{y_{i}}{x_{i}} - u)$, let $x_{i+1} := x_{i}$ and $y_{i+1} := \frac{y_{i}}{x_{i}} - u$.}
			\item[2.]{If $v(x_{i}) > v(\frac{y_{i}}{x_{i}} - u)$, let $x_{i+1} := \frac{y_{i}}{x_{i}} - u$ and $y_{i+1} := x_{i}$.}
		\end{enumerate}}

\end{enumerate}

Without loss of generality we may assume that $x_{i}(\alpha) > 0$ and $y_{i}(\alpha) > 0$. Since $\bqf{\alpha^{(i)}}{\beta^{(i)}} \subsetneq \m^{(i)}$ for all $i < r$, we also have $x_{i}(\beta) > 0$ if $i < r$.

\begin{proposition}

	Suppose $A$ is the localization at the maximal ideal $(\enum{z})$ of a finitely generated two-dimensional regular $R$-algebra $R[\enum{z}]$ without zero divisors, where $R$ is a real closed field. Let $\alpha,\beta \in \Sper{A}$ both centered at the maximal ideal $\m$ of $A$ such that $\bqf{\alpha}{\beta} \subsetneq \m = \sqrt{\bqf{\alpha}{\beta}}$. Let $v := v_{\alpha}$. Suppose $0 < v(z_{1}) \le v(z_{2}), \ldots, v(z_{n})$. Then there exist elements $u_{2}, \ldots, u_{n} \in R$ such that the quadratic transform $A'$ of $A$ along $v$ equals the localization of $R[\enum{z'}]$ at the maximal ideal $(\enum{z'})$, where $z_{1}' := z_{1}$ and $z_{j}' = \frac{z_{j}}{z_{1}} - u_{j}$ if $j > 1$. Suppose further that $(z_{1},z_{2})$ is a regular system of parameters of $A$, then $(z_{1}',z_{2}')$ is a regular system of parameters of $A'$.

\label{FinGenQT}
\end{proposition}

\begin{proof}

	Let $j > 1$. If $v(z_{j}) > v(z_{1})$, let $u_{j} := 0$.  If $v(z_{j}) = v({z_{1}})$, then, since $\m / \m^{\alpha} \cong R$ (as shown in the proof of \ref{AJMQuadTrans}), there exists some $u \in R$ such that $v(z_{j} - uz_{1}) > v(\m) = v(z_{1})$, and we take $u_{j} := u$.
	
	Let $B := A[z_{1}^{-1} \m] = \set{\frac{a}{z_{1}^{m}}}{a \in A, \ \ord{A}(a) \ge m}$. Then we have $A' = S^{-1} B$, where $S = \set{b \in B}{v(b) = 0}$. Let $a \in A$. By assumption $a = \frac{a_{1}}{a_{2}}$, where $a_{1} \in R[\enum{z}]$ and $a_{2} \in R[\enum{z}] \setminus (\enum{z})$. Let $m \in \N$ such that $m \le \ord{A}(a) = \ord{A}(a_{1})$, hence $a_{1}' := \frac{a_{1}}{z_{1}^{m}} \in R[z_{1}, \frac{z_{2}}{z_{1}}, \ldots, \frac{z_{n}}{z_{1}}] = R[\enum{z'}]$. Further $a_{2} \in R[\enum{z}] \setminus (\enum{z}) \subset R[\enum{z'}] \setminus (\enum{z'})$. Hence $B \subset R[\enum{z'}]_{(\enum{z'})}$. The valuation $v$ extends uniquely to $R[\enum{z'}]$, it is non-negative on this ring, and its center is the maximal ideal $(\enum{z'})$, therefore it also extends uniquely to $R[\enum{z'}]_{(\enum{z'})}$
	
	Suppose now that $v(\frac{a}{z_{1}^{m}}) = 0$, i.e., $v(a_{1}) = v(a) = m v(z_{1})$. Then we have that $v(a_{1}') = v(a_{2}) = 0$. Since $v$ is centered on $(\enum{z'})$ in $R[\enum{z'}]$, we have that $a_{1}' \in R[\enum{z'}] \setminus (\enum{z'})$. Thus, we have shown that $A' = S^{-1} B = R[\enum{z'}]_{(\enum{z'})}$.
	
	The last assertion follows immediately from the considerations we made above.

\end{proof}

\begin{lemma}

	Let $\alpha,\beta \in \Sper{A}$ such that $\cent{\alpha} = \m = \cent{\beta}$ and $\bqf{\alpha}{\beta} \subsetneq \m = \sqrt{\bqf{\alpha}{\beta}}$. Let $(A,\m,R) =: (A^{(0)},\m^{(0)},R) \prec \cdots \prec (A^{(r)},\m^{(r)},R)$ be the sequence of quadratic transformations along $v := v_{\alpha}$ such that $T^{(r)}(\bqf{\alpha}{\beta}) = \m^{(r)} = \bqf{\alpha^{(r)}}{\beta^{(r)}}$. Then every $g \in A \setminus \bqf{\alpha}{\beta}$ has the form $x_{r}^{e} y_{r}^{f} w$, where $(x_{r},y_{r})$ is a regular system of parameters of $A^{(r)}$, $w \in A^{(r)} \setminus \m^{(r)}$, $e,f \in \N$, and $e = 0$ (resp. $f = 0$) if $x_{r}$ (resp. $y_{r}$) changes sign between $\alpha$ and $\beta$.

\label{LocUni}
\end{lemma}

\begin{proof}

	Let $g \in A \setminus \bqf{\alpha}{\beta}$. Let $I := \set{a \in A}{v_{\alpha}(a) \ge v_{\alpha}(g)} = \set{a \in A}{v_{\beta}(a) \ge v_{\beta}(g)}$. Since $g \notin \bqf{\alpha}{\beta}$, we have that $T^{(r)}(I) = A^{(r)}$, by Lemma \ref{IdPropTrans}.\\
	Let $\nu_{0} = \ord{A}{I}$ be the order of the ideal $I$. Since $g$ has minimal value in $I$, the element $g \cdot x_{0}^{-\nu_{0}}$ has minimal value in the transform $T(I)$ of $I$. By induction, one shows that there exist $\nu_{i} \in \N$ ($0 \le i < r$) such that $g \cdot x_{0}^{-\nu_{0}} \cdots x_{r-1}^{-\nu_{r-1}}$ has minimal value in $T^{(r)}(I) = A^{(r)}$, hence
	\[ g = x_{0}^{\nu_{0}} \cdots x_{r-1}^{\nu_{r-1}} w' \]
	for some  for  and some unit $w'$ in $A^{(r)}$.
	
	We would like to write $g$ in the form $x_{r}^{e} y_{r}^{f} w$, where $w \in{ A^{(r)}}^{\times}$, $e,f \in \N$, and $e = 0$ (resp. $f = 0$) if $x_{r}$ (resp. $y_{r}$) changes sign between $\alpha$ and $\beta$. In order to see how this can be done, we have to know for all $i < r$ how to represent a product $x_{i}^{s} y_{i}^{t}$ in terms of $x_{i+1}$ and $y_{i+1}$. We will look again at the several cases of a quadratic transformation along $v$.
	
	\begin{enumerate}
	
		\item[I.1]{If $x_{i} = x_{i+1}$ and $y_{i} = x_{i+1}y_{i+1}$, then $x_{i}^{s} y_{i}^{t} = x_{i+1}^{s+t} y_{i+1}^{t}$.}
		\item[I.2]{If $x_{i} = y_{i+1}$ and $y_{i} = x_{i+1}y_{i+1}$, then $x_{i}^{s} y_{i}^{t} = x_{i+1}^{t} y_{i+1}^{s+t}$.}
		\item[II.1]{If $x_{i} = x_{i+1}$ and $y_{i} = x_{i+1}(y_{i+1} + u)$ for some $u \in {A^{(i)}}^{\times}$, then we have $x_{i}^{s} y_{i}^{t} = x_{i+1}^{s+t} (y_{i+1} + u)^{t}$. Note that $y_{i+1} + u$ is a unit in $A^{(i+1)}$.}
		\item[II.2]{If $x_{i} = y_{i+1}$ and $y_{i} = y_{i+1}(x_{i+1} + u)$ for some $u \in {A^{(i)}}^{\times}$, then we have $x_{i}^{s} y_{i}^{t} = y_{i+1}^{s+t} (x_{i+1} + u)^{t}$, and $x_{i+1} + u$ is a unit in $A^{(i+1)}$.}
	
	\end{enumerate}
	
	Let $0 \le i < r$. If both $x_{i}$ and $y_{i}$ are positive at $\alpha$ and $\beta$ and we are in case I.1 or I.2, then $x_{i+1}$ and $y_{i+1}$ are also positive at $\alpha$ and $\beta$. If we are in case II.1 or II.2, it is possible that $\frac{y_{i}}{x_{i}} - u$ changes sign between $\alpha$ and $\beta$. In case II.2, we then have $x_{i+1} = \frac{y_{i}}{x_{i}} - u \in \bqf{\alpha^{(i+1)}}{\beta^{(i+1)}}$, and, since $x_{i+1}$ has minimal value in $\m^{(i+1)}$ and $\m^{(i+1)}$ contains $\bqf{\alpha^{(i+1)}}{\beta^{(i+1)}}$, this yields $\m^{(i+1)} = \bqf{\alpha^{(i+1)}}{\beta^{(i+1)}}$, thus $i+1 = r$. In case II.1, if $v(x_{i+1}) = v(x_{i}) = v(\frac{y_{i}}{x_{i}} - u) = v(y_{i+1})$ and $y_{i+1} = \frac{y_{i}}{x_{i}} - u \in \bqf{\alpha^{(i+1)}}{\beta^{(i+1)}}$, we can use the same arguments as in the last sentence to show that $i+1 = r$. In the same case, if $v(x_{i}) < v(\frac{y_{i}}{x_{i}} - u)$ and $y_{i+1} = \frac{y_{i}}{x_{i}} - u \in \bqf{\alpha^{(i+1)}}{\beta^{(i+1)}}$, we will reach $A^{(r)}$ after a series of I.1 transformations and maybe one last I.2 transformation. The same holds if $y_{0}$ changes sign between $\alpha$ and $\beta$.
	
	We will now consider the implications of the last considerations for the representation of an element $g \in A \setminus \bqf{\alpha}{\beta}$ in terms of $x_{r}$ and $y_{r}$ in the ring $A^{(r)}$.
	
	1. Suppose, for all $i \le r$, both $x_{i}$ and $y_{i}$ are positive at $\alpha$ and $\beta$. Then, in the representation $g = x_{r}^{e} y_{r}^{f} w$ with $e,f \in \N$ and $w \in{ A^{(r)}}^{\times}$ of some element $g \in A \setminus \bqf{\alpha}{\beta}$, we do not care whether $e$ or $f$ is zero or not.
	
	2. Now suppose $y_{j}$ changes sign between $\alpha$ and $\beta$ for some $j$ and $j$ is minimal with this property. Hence, if $j > 0$, the last quadratic transformation from $A^{(j-1)}$ to $A^{(j)}$ must be of type II.1. In particular, if $j > 0$, $x_{j} = x_{j-1}$ does not change its sign between $\alpha$ and $\beta$, and we have $x_{0}^{\nu_{0}} \cdots x_{j-1}^{\nu_{j-1}} = x_{j}^{s} (y_{j} + u)^{t}$, where $s,t \in \N$, $u$ a unit in $A^{(j-1)}$, and $y_{j} + u$ is a unit in $A^{(j)}$.\\
	As mentioned above, we have $j = r$ or we will reach $A^{(r)}$ after a series of I.1 transformations and maybe one last I.2 transformation. Thus, if $j < r$, we have $x_{j} = x_{j+1} = \cdots = x_{r-1}$ and each $g \in A \setminus \bqf{\alpha}{\beta}$ has the form $x_{r}^{e} w$ or $y_{r}^{f} w$ with $w \in{ A^{(r)}}^{\times}$, depending on whether $x_{r-1} = x_{r}$ (I.1) or $x_{r-1} = y_{r}$ (I.2). Note that in these representations none of the factors changes sign between $\alpha$ and $\beta$. If $j = r$, then $x_{r-1} = x_{r}$, and therefore each $g \in A \setminus \bqf{\alpha}{\beta}$ has the form $x_{r}^{e} w$ with $w \in{ A^{(r)}}^{\times}$, and $x_{r}$ does not change its sign.
	
	3. Suppose that, for all $i \le r$, $y_{i}$ does not change sign between $\alpha$ and $\beta$, but $x_{r}$ changes sign between these two points. Then, as seen above, the last transformation must be of type II.2, hence $g = y_{r}^{f} w$ with $w \in{ A^{(r)}}^{\times}$.
	
	Note that the last discussion also showed that at least one of the elements $x_{r}$ and $y_{r}$ does not change its sign between $\alpha$ and $\beta$.

\end{proof}

Finally, we are able to prove our main result.

\begin{theorem}

	Let $W$ be a non-singular two-dimensional affine real algebraic variety over a real closed field $R$, and let $\mathcal{P}(W)$ be its coordinate ring. Then every pair of points $\alpha,\beta \in \Sper{\mathcal{P}(W)}$ having a common center satisfies the connectedness condition. In particular, $\mathcal{P}(W)$ is Pierce-Birkhoff.

\end{theorem}

\begin{proof}

	We may assume that $W$ is irreducible. Let $R[x,y,\enumi{z}{m}]$ be the coordinate ring $\mathcal{P}(W) = R[X,Y,\enumi{Z}{m}] / \mathcal{I}(W)$ of $W$, where $R[X,Y,\enumi{Z}{m}]$ is the polynomial ring in $m + 2$ indeterminates over $R$, and $\mathcal{I}(W)$ the prime ideal of $W$ in $R[X,Y,\enumi{Z}{m}]$. Then $\mathcal{P}(W)$ is a finitely generated two-dimensional regular $R$-algebra without zero -divisors. Suppose $x$ and $y$ are algebraically independent over $R$, and that $z_{j} \notin R$ for all $j \in \{ 1,\ldots, m \}$.
	
	Let $\alpha,\beta \in \Sper{\mathcal{P}(W)}$, and assume that they are distinct and have a common center. If the height of $\sqrt{\bqf{\alpha}{\beta}}$ is one, we have shown in Lemma \ref{CCone} that $\alpha$ and $\beta$ satisfy the connectedness condition. Hence, we can suppose that the height of $\sqrt{\bqf{\alpha}{\beta}}$ is two, i.e., $\sqrt{\bqf{\alpha}{\beta}} = \cent{\alpha} = \cent{\beta}$ is a real maximal ideal $\mathcal{M}$ of $\mathcal{P}(W)$, and the residue field $\mathcal{P}(W) / \mathcal{M}$ equals $R$. Therefore, we may assume that $\mathcal{M} = (x,y,\enumi{z}{m})$.
	
	Let $\enumi{g}{s} \in A \setminus \bqf{\alpha}{\beta}$, and let $A$ be the localization of $\mathcal{P}(W)$ at $\mathcal{M} = \sqrt{\bqf{\alpha}{\beta}}$. Let $\bqf{\alpha}{\beta}_{A}$ be the separating ideal of (the unique extensions of) $\alpha$ and $\beta$ in $A$. Then $\bqf{\alpha}{\beta}_{A} \cap \mathcal{P}(W) = \bqf{\alpha}{\beta}$. Hence $\enumi{g}{s} \notin \bqf{\alpha}{\beta}_{A}$.
	
	Let $(A,\m,R) =: (A^{(0)},\m^{(0)},R) \prec \cdots \prec (A^{(r)},\m^{(r)},R)$ be the sequence of quadratic transformations along $v_{\alpha}$ such that $\m^{(r)} = T^{(r)}(\bqf{\alpha^{(0)}}{\beta^{(0)}}) = \m^{(r)} = \bqf{\alpha^{(r)}}{\beta^{(r)}}$. By Proposition \ref{FinGenQT}, $A^{(r)} = R[x_{r},y_{r},\enumi{z^{(r)}}{m}]_{(x_{r},y_{r},\enumi{z^{(r)}}{m})}$ for some regular system of parameters $(x_{r},y_{r})$ of $A^{(r)}$, and $\mathcal{P}(W) \subset R[x_{r},y_{r},\enumi{z^{(r)}}{m}]$. Recall that $x_{r}$ and $y_{r}$ are algebraically independent over $R$, and we assumed that $x_{r}(\alpha) > 0$ and $y_{r}(\alpha) > 0$.
	
	Then, by Lemma \ref{LocUni}, for all $j \in \{ 1,\ldots,s \}$, $g_{j}$ has the form $x_{r}^{e_{j}} y_{r}^{f_{j}} w_{j}$, where $w_{j} \in A^{(r)} \setminus \m^{(r)}$, and $e_{j},f_{j} \in \N$ are such that $e_{j} = 0$ (resp. $f_{j} = 0$) if $x_{r}$ (resp. $y_{r}$) changes sign between $\alpha$ and $\beta$. The units $w_{j}$ can be written as $\frac{w_{j1}}{w_{j2}}$ with $w_{j1},w_{j2} \in R[x_{r},y_{r},\enumi{z^{(r)}}{m}] \setminus (x_{r},y_{r},\enumi{z^{(r)}}{m})$. Clearly, $0 \notin \mathcal{Z}(\prod\limits_{j = 1}^{s} w_{j1}w_{j2})$, the zero set of this element.
	
	Consider the polynomial ring $R[X,Y,\enumi{Z}{m}]$. Since, $\enumi{z^{(r)}}{m}$ are algebraic over $R(x_{r},y_{r})$, there exists a surjective ring homomorphism $\pi$ from $R[X,Y,\enumi{Z}{m}]$ to $R[x_{r},y_{r},\enumi{z^{(r)}}{m}]$ such that $\mathcal{J} := \pi^{-1}(0)$ is a prime ideal, and $\pi(X) = x_{r}$, $\pi(Y) = y_{r}$ and $\pi(Z_{j}) = z_{j}^{(r)}$ for all $j \in \{ 1, \ldots, m \}$. Let $W^{(r)}$ be the irreducible affine real algebraic variety over $R$ corresponding to $R[x_{r},y_{r},\enumi{z^{(r)}}{m}]$. Then $0$ is a regular point of $W^{(r)}$ with $(x_{r},y_{r})$ as a regular system of local parameters, hence there exist $\enumi{f}{m} \in \mathcal{I}(W^{(r)}) = \mathcal{J} \subset R[X,Y,\enumi{Z}{m}]$ such that the matrix $\left( \frac{\partial f_{j}}{\partial Z_{k}} \right)_{j,k = 1, \ldots, m}$ is invertible. By the semialgebraic implicit function theorem, there exist an open semialgebraic neighborhood $U$ of $0$ in $R^{2}$, an open semialgebraic neighborhood $V$ of $0$ in $R^{m}$ and a continuous semialgebraic map $\varphi\colon U \rightarrow V$ such that $\varphi(0) = 0$ and
	\[ f_{1}(a,b) = \cdots = f_{m}(a,b) = 0 \iff \varphi(a) = b \]
	for every $(a,b) \in U \times V$. Then the map $\psi\colon U \mapsto \mathcal{Z}(\enumi{f}{m}) \cap U \times V$, $a \mapsto (a,\varphi(a))$ is a semialgebraic continuous bijection.
	
	Let $\mathcal{U}$ be an open neighborhood of $0$ in $R^{m+2}$ such that $\mathcal{U} \subset U \times V$, $\mathcal{Z}(\enumi{f}{m}) \cap \mathcal{U} = W^{(r)} \cap \mathcal{U}$, and $\mathcal{Z}(\prod\limits_{j = 1}^{s} w_{j1}w_{j2}) \cap \mathcal{U} = \emptyset$. Let $\varepsilon \in R$ such that $\varepsilon > 0$ and
	\[ \set{t \in R^{2}}{(\varepsilon - x_{r}^{2} - y_{r}^{2})(t) > 0} \subset \psi^{-1}(\mathcal{U}). \]
	Let $I$ be the subset of the two-element set $\{ x_{r},y_{r} \}$ that contains all elements which do not change sign between $\alpha$ and $\beta$, and let $\mathcal{V} := \set{t \in R^{2}}{\varepsilon - \sum_{i = 1}^{2} t_{i}^{2} > 0, \ p(t) > 0 \ (p \in \{ X,Y \}, \ \pi(p) \in I)} \subset \psi^{-1}(\mathcal{U})$. Note that from the proof of Lemma \ref{LocUni}, it follows that $I$ is not empty. By Proposition \ref{ConvSemiConn}, we have that $\mathcal{V}$ is semialgebraically connected, and therefore, $\psi(\mathcal{V}) \subset W^{(r)} \cap \mathcal{U}$ is also semialgebraically connected (Proposition \ref{ContSemiConn}). Hence, by Proposition \ref{SemiConnSpec}, $D := \tilde{\psi(\mathcal{V})}$ is connected in $\Sper{\mathcal{P}(W^{(r)})}$, and $D$ contains $\alpha^{(r)}$ and $\beta^{(r)}$, since $v_{\alpha}(x_{r}), v_{\alpha}(y_{r}), v_{\beta}(x_{r})$ and $v_{\beta}(y_{r})$ are all positive, and since we assumed that $x_{r}(\alpha) > 0$ and $y_{r}(\alpha) > 0$, and therefore they are also positive at $\beta$ if they do not change sign.
	
	Let $j \in \{ 1, \ldots, s \}$. We have $w_{j2}g_{j} = x_{r}^{e_{j}} y_{r}^{f_{j}} w_{j1}$, with $e_{j} = 0$ if $x_{r} \notin I$ and $f_{j} = 0$ if $y_{r} \notin I$, and therefore $g_{j}(\delta') \neq 0$ for all $\delta' \in D$. Consider the natural continuous map $\Sper{}(\pi)\colon \Sper(\mathcal{P}(W^{(r)})) \rightarrow \Sper{\mathcal{P}(W)}$, and set $C := \Sper{}(\pi)(D) = \set{\delta' \cap \mathcal{P}(W)}{\delta' \in D}$. By Proposition \ref{SemiConnSpec}, $D$ is connected, thus $C$ must also be connected. $D$ contains $\alpha^{(r)}$ and $\beta^{(r)}$, and therefore $C$ contains $\alpha$ and $\beta$. Since $g_{j}(\delta') \neq 0$ for all $j \in \{ 1, \ldots, s \}$ and all $\delta' \in D$, we have $C \cap \set{\delta \in \Sper{\mathcal{P}}(W)}{g_{j}(\delta) = 0} = \emptyset$.

\end{proof}

Sven Wagner\\
Universit\"at Konstanz\\
Fachbereich Mathematik und Statistik\\
78457 Konstanz, Germany\\
sven.wagner@uni-konstanz.de

\end{document}